\documentclass[11pt,leqno]{article}
\usepackage{graphicx, amsfonts, amsthm, amsxtra, amssymb, verbatim,bbm}
\usepackage[mathscr]{euscript}
\usepackage{hyperref}
\begin{document}

\newcommand{\mE}{\mathcal{E}}
\newcommand{\be}{\begin{eqnarray}}
\newcommand{\ee}{\end{eqnarray}}
\newcommand{\bea}{\begin{eqnarray}}
\newcommand{\eea}{\end{eqnarray}}
\newcommand{\cL}{\mathcal{L}}
\newcommand{\E}[1]{{E}_{#1}}
\newcommand{\dq}{q \frac{\partial}{\partial q}}
\newcommand{\p}{\partial}
\newcommand{\ib}{\bar{\imath}}
\newcommand{\jb}{\bar{\jmath}}
\newcommand{\lb}{\bar{l}}
\newcommand{\mb}{\bar{m}}
\newcommand{\kb}{\bar{k}}
\newcommand{\bb}{\bar{b}}

\def\tgtwo{\mathfrak{t}}
\def\tgone{\mathfrak{t}}
\def\tgzero{\mathfrak{t}}
\def\kgone{\mathfrak{k}}
\def\ggtwo{\mathfrak{g}}
\def\ggzero{\mathfrak{g}}

\def\liealg{{\mathfrak G}}    

\def\Fg{{\sf F}}     
\def\hol{{\rm hol}}  
\def\non{{\rm non}}  
\def\alg{{\rm alg}}  

\def\Z{\mathbb{Z}}                   
\def\Q{\mathbb{Q}}                   
\def\C{\mathbb{C}}                   
\def\N{\mathbb{N}}                   
\def\uhp{{\mathbb H}}                
\def\A{\mathbb{A}}                   
\def\dR{{\rm dR}}                    
\def\F{{\cal F}}                     
\def\Sp{{\rm Sp}}                    
\def\Gm{\mathbb{G}_m}                 
\def\Ga{\mathbb{G}_a}                 
\def\Tr{{\rm Tr}}                      
\def\tr{{{\mathsf t}{\mathsf r}}}                 
\def\k{{\sf k}}                     
\def\ring{{\sf R}}                   
\def\X{{\sf X}}                      
\def\T{{\sf T}}                      
\def\Ts{{\sf S}}
\def\cmv{{\sf M}}                    
\def\BG{{\sf G}}                       
\def\podu{{\sf pd}}                   
\def\ped{{\sf U}}                    
\def\per{{\sf  P}}                   
\def\gm{{\sf  A}}                    
\def\ben{{\sf b}}                    

\def\Rav{{\mathfrak M }}                     
\def\Ram{{\mathfrak C}}                     
\def\Rap{{\mathfrak G}}                     

\def\nov{{\sf  n}}                    
\def\mov{{\sf  m}}                    
\def\Yuk{{\sf C}^\alg}                 
\def\YukA{{\sf C}}                     
\def\Ra{{\sf R}}                      
\def\hn{{\sf h}}                      
\def\cpe{{\sf C}}                     
\def\g{{\sf g}}                       
\def\t{{\sf t}}                       
\def\pedo{{\sf  \Pi}}                  

\def\Der{{\rm Der}}                   
\def\MMF{{\sf MF}}                    
\def\codim{{\rm codim}}                
\def\dim{{\rm    dim}}                
\def\Lie{{\rm Lie}}                   
\def\gg{{\mathfrak g}}                

\def\u{{\sf u}}                       

\def\imh{{  \Psi}}                 
\def\imc{{  \Phi }}                  
\def\stab{{\rm Stab }}               
\def\Vec{{\rm Vec}}                 
\def\prim{{\rm prim}}                  

\def\Fg{{\sf F}}     
\def\hol{{\rm hol}}  
\def\non{{\rm non}}  
\def\alg{{\rm alg}}  

\def\bcov{{\rm \Q(\T)}}

\def\spec{{\rm Spec}}            
\def\ker{{\rm ker}}              
\def\GL{{\rm GL}}                

\def\O{{\cal O}}                     

\def\Mat{{\rm Mat}}              

\newtheorem{theo}{Theorem}
\newtheorem{exam}{Example}
\newtheorem{coro}{Corollary}
\newtheorem{defi}{Definition}
\newtheorem{prob}{Problem}
\newtheorem{lemm}{Lemma}
\newtheorem{prop}{Proposition}
\newtheorem{rem}{Remark}
\newtheorem{conj}{Conjecture}
\newtheorem{calc}{}

\begin{center}
{\LARGE\bf Gauss-Manin connection in disguise: \\ Calabi-Yau threefolds
\footnote{ 
Math. classification: 14N35, 
14J15, 32G20
}}
\\
\vspace{.25in} {\large {\sc Murad Alim, Hossein Movasati, Emanuel Scheidegger, \\ Shing-Tung Yau}}
\\

\end{center}

\begin{abstract}
 We describe a Lie Algebra on the moduli space of  Calabi-Yau threefolds enhanced with differential forms and its relation to the Bershadsky-Cecotti-Ooguri-Vafa holomorphic anomaly equation. In particular, 
we describe algebraic topological string partition functions $\Fg_g^\alg,\ g\geq 1$, which encode the polynomial structure of holomorphic and non-holomorphic 
topological string partition functions. Our approach is based on Grothendieck's algebraic de Rham cohomology and on the algebraic Gauss-Manin connection.
In this way, we recover a result
of Yamaguchi-Yau and Alim-L\"ange in an algebraic context. 
Our proofs use the fact that the special polynomial generators defined using the special 
geometry of deformation spaces of Calabi-Yau threefolds correspond to coordinates on such a moduli space. We discuss the mirror quintic as an example.
\end{abstract}

\tableofcontents


\section{Introduction}
Mirror symmetry identifies deformation families of Calabi-Yau (CY) threefolds. It originates from two dimensional sigma models into CY target spaces $\check{X}$ and $X$ and two equivalent twists which give the A- and the B-model and which probe the symplectic and complex geometry of $\check{X}$ and $X$ respectively \cite{Witten:1988xj,Witten:1991zz}. 

Mirror symmetry is a rich source of far-reaching predictions, especially regarding the enumerative geometry of maps from genus $g$ Riemann surfaces into a CY threefold $\check{X}$. The predictions are made by performing computations on the B-model side which sees the deformations of complex structure of the mirror CY $X$. The non-trivial step, which is guided by physics, is to identify the equivalent structures on the A-model side and match the two.

The first enumerative predictions of mirror symmetry at genus zero were made by Candelas, de la Ossa, Green and Parkes in Ref.~\cite{Candelas:1990rm}, higher genus predictions were put forward by Bershadsky, Cecotti, Ooguri and Vafa (BCOV) in Refs.~\cite{Bershadsky:1993ta,Bershadsky:1993cx}. To prove these predictions and formulate them rigorously is a great mathematical challenge.

The formulation of the moduli spaces of stable maps by Kontsevich \cite{Kontsevich:hms} provided a mathematical formulation of the $A$-model and a check of many results confirming the predictions of mirror symmetry. The computations of Ref.~\cite{Candelas:1990rm} for genus zero Gromov-Witten invariants  
were put in a Hodge theoretical context by Morrison in Ref.~\cite{Morrison:1992}. Genus zero mirror symmetry can now be understood as matching two different variations of Hodge structure associated to $\check{X}$ and $X$, See Refs.~\cite{Morrison:rev,Cox:1999,Voisin:MS}.

Mirror symmetry at higher genus remains challenging both computationally and conceptually. A fruitful way to think about higher genus mirror symmetry is through geometric quantization as proposed by Witten in Ref.~\cite{Witten:1993ed}.  A mathematical reformulation of BCOV for the B-model was put forward by Costello and Li in Ref.~\cite{Costello:2012cy}.

In the present work we will follow a different approach and put forward a new algebraic framework to formulate higher genus mirror symmetry where we can work over an arbitrary field of characteristic zero. Our approach is based on Grothendieck's algebraic de Rham cohomology and Katz-Oda's algebraic Gauss-Manin connection. It further builds on results of Yamaguchi-Yau \cite{Yamaguchi:2004bt} and Alim-L\"ange\cite{Alim:2007qj}, who uncovered a polynomial structure of higher genus topological string partition functions. This polynomial structure is based on the variation of Hodge structure at genus zero and puts forward variants of BCOV equations which can be understood in a purely holomorphic context. 

In the algebraic context, surprisingly, no reference to periods or variation of Hodge structures is needed,  as all these are hidden in the so called \emph{Gauss-Manin connection
in disguise}.\footnote{The terminology arose from a private letter of Pierre Deligne to the second author~\cite{Deligne:letter}.} The new way of looking Gauss-Manin connection was studied by the second author, see Ref.~\cite{ho14} for elliptic curve case,  Ref.~\cite{ho22} for mirror quintic case and Ref.~\cite{ho18} for a general framework. 
 The richness of this point of view 
is due to its base space which is 
the moduli space of varieties of a fixed topological type and enhanced with differential forms. Computations on such moduli spaces were
already implicitly in use by Yamaguchi-Yau \cite{Yamaguchi:2004bt}  and Alim-L\"ange \cite{Alim:2007qj} without referring to the moduli space itself, however, its introduction 
and existence in algebraic geometry for special cases go back to the works of the second author.  
Such moduli spaces give a natural
framework for dealing with both automorphic forms and topological string partition functions. 
In the case of elliptic curves \cite{ho14},  the theory of modular and quasi-modular forms is recovered. In the case of compact Calabi-Yau threefolds we obtain new types of  functions  which transcend the world of automorphic forms. In the present text we develop the algebraic structure for any CY threefold. As an example, we study the mirror quintic in detail.

In the following, we recall the basic setting of Refs.~\cite{ho18, ho22}. For a background in Hodge theory and algebraic de Rham cohomology we refer to Grothendieck's original article \cite{Grothendieck:1966} or Deligne's lecture notes in \cite{Deligne:lecture}. Let $\k$ be a field of characteristic zero small enough so that it can be embedded in $\C$. For a Calabi-Yau threefold $X$ defined over $\k$ let $H_\dR^3(X)$ be  the third algebraic de Rham cohomology of $X$ and
$$
0=F^4\subset F^3\subset F^2\subset F^1\subset F^0=H_\dR^3(X)\,,
$$
be the corresponding  Hodge filtration. The intersection form  in $H^3_\dR(X)$ is defined to be
\begin{equation}
\label{16s2014}
H_\dR^3(X)\times H_\dR^3(X)\to \k,\ \ \langle\omega_1,\omega_2\rangle=\Tr(\omega_1\cup \omega_2):= \frac{1}{(2\pi i)^3}\int_{X}\omega_1\cup \omega_2\,.
\end{equation}
All the structure above, that is, the de Rham cohomology, its Hodge filtration and intersection form, is also defined over $\k$, that is they do not depend on the embedding $\k\hookrightarrow \C$, 
see for instance Deligne's lecture notes \cite{Deligne:lecture}.  
Let $\hn=\dim(F^2)-1$  and let $\Phi$ be the following constant matrix:
\begin{equation}
\label{31aug10}
\Phi:=
\begin{pmatrix}
0&0&0&-1\\
0& 0&\mathbbm{1}_{\hn \times \hn}&0\\
0& -\mathbbm{1}_{\hn\times \hn}&0&0\\
 1&0&0&0
\end{pmatrix}.
\end{equation}
Here, we use $(2\hn+2)\times (2\hn+2)$ block matrices according to the decomposition $2\hn+2=1+\hn+\hn+1$ and 
$\mathbbm{1}_{\hn\times \hn}$ denotes the $\hn\times \hn$ identity matrix.  
The following definition is taken from Ref.~\cite{ho18}. 
An enhanced Calabi-Yau threefold is a
pair $(X,[\omega_1,\omega_2,\ldots,\omega_{2\hn+2}])$, 
where $X$ is as before and  $\omega_1,\omega_2,\ldots,\omega_{2\hn+2}$ 
is a basis of $H^3_\dR(X)$. 
We choose the basis such that  
\begin{enumerate}
\item
It is compatible with the Hodge filtration, that is, $\omega_1\in F^3$, $\omega_1, \omega_2,\ldots,\omega_{\hn+1}\in F^2$, $\omega_1, \omega_2,\ldots,\omega_{2\hn+1} \in F^1$ and  $ \omega_1, \omega_2,\ldots,\omega_{2\hn+2}  \in F^0$. 
\item 
The intersection form in this basis is the constant matrix $\imc$: 
\begin{equation}
\label{intmatrix}
[\langle \omega_i,\omega_j\rangle]=\Phi.
\end{equation}
\end{enumerate}
Let $\T$  be the moduli of enhanced Calabi-Yau threefolds of a 
fixed topological type. 
The algebraic group 
\begin{equation}
\BG:=\left \{\g\in \GL(2\hn+2,\k)\mid \g \text{ is block upper triangular and  }  \g^\tr\imc\g=\imc\ \ \right \}
\end{equation}
acts from the right on $\T$ and its Lie algebra plays an important role in our main theorem. 
\begin{equation}
\Lie(\BG)=
\left \{\gg\in \Mat(2\hn+2,\k)\mid \gg \text{ is block upper triangular and  }  \gg^\tr\imc+\imc\gg=0\ \ \right \}.
\end{equation}
Here, by block triangular we mean triangular with respect to the partition 
$2\hn+2=1+\hn+\hn+1$. We have
$$
\dim(\BG)=\frac{3\hn^2+5\hn+4}{2},\ \ \ \dim(\T)=\hn+\dim(\BG). 
$$
Special geometry and period manipulations suggest that 
$\T$ has a canonical structure of an affine variety  over $\bar\Q$
and the action of $\BG$ on $\T$ is algebraic.
We have a universal family  $\X/\T$ of Calabi-Yau threefolds and 
by our construction we have elements $\tilde \omega_i\in H^3(\X/\T)$ such that $\tilde \omega_i$ restricted
to the fiber $\X_\t$ is the chosen $\omega_i\in H^3_\dR(\X_\t)$.  
For simplicity we write $\tilde\omega_i=\omega_i$. Here, $H^3(\X/\T)$ denotes the set of 
global sections of the relative third de Rham cohomology of $\X$ over $\T$.  Furthermore, 
there is an affine variety $\tilde\T$  such that $\T$ is a Zarski open subset of $\tilde \T$, 
the action of $\BG$ on $\T$ extends to $\tilde\T$  and 
the quotient $\tilde\T/\BG$ is a projective variety (and hence compact).  
All the above statements can be verified for instance for mirror quintic Calabi-Yau threefold, see \S\ref{mqc}. 
Since we have now a good understanding of the classical moduli of Calabi-Yau varieties both in complex and algebraic context 
(see respectively Ref.~\cite{Viehweg:1995} and Ref.~\cite{Todorov:2003}), verifying the above statements are not out of reach. 
For the purpose of the present text either assume that
the universal family $\X/\T$ over $\bar\Q$ exists or assume that  $\T$ is the total space of the 
choices of the basis $\omega_i$ over a local patch of the moduli space of $X$. 
By Bogomolov-Tian-Todorov Theorem such a moduli space is smooth. 
We further assume that
in a local patch of moduli space, the universal family of Calabi-Yau threefolds $X$ exists and  no Calabi-Yau threefold $X$ in such a local patch has an isomorphism which 
acts non-identically on $H^3_\dR(X)$.  
In the last case one has to replace all the algebraic notations 
below by their holomorphic counterpart. Let
$$
\nabla: H^3_\dR(\X/\T)\to \Omega^{1}_\T\otimes_{\O_\T} H^3_\dR(\X/\T)\,,
$$
be the algebraic Gauss-Manin connection of the family $\X/\T$ due to Katz-Oda \cite{Katz:1968}, where 
$\O_\T$ is  the $\bar \Q$-algebra of regular  functions on $\T$ and $\Omega^1_\T$ is the $\O_\T$-module of differential $1$-forms in $\T$.  
For any vector field $\Ra$ in $\T$, let $\nabla_\Ra$ be the Gauss-Manin connection 
composed with the vector field $\Ra$. We write $$
\nabla_\Ra\omega=\gm_\Ra\omega,
$$
where $\gm_\Ra$ is a $(2\hn+2)\times (2\hn+2)$ matrix with entries in $\O_\T$ and   $\omega:=[\omega_1,\ldots,\omega_{2\hn+2}]^{\tr}$. 

To state our main theorem we will introduce some physics notation which will be useful in the rest of the paper. We split the notation for the basis $\omega$ in the following way $[\omega_1,\omega_2,\ldots,\omega_{2\hn+2}]=[\alpha_0,\alpha_i,\beta^i,\beta^0]\,, i=1,2,\ldots ,\hn$. The distinction between upper and lower indices here does not yet carry particular meaning. They are chosen such that they are compatible with the physics convention of summing over repeated upper and lower indices. We will write out matrices in terms of their components, 
denoting by an index $i$ the rows and an index $j$ the columns. We further introduce $\delta_i^j$ which is $1$ when $i=j$ and zero otherwise.

\begin{theo}
\label{maintheo}
We have the following
\begin{enumerate}
\item
There are unique vector fields $\Ra_k, \ k=1,2,\ldots,\hn$ in $\T$ and unique 
$\Yuk_{ijk}\in \O_\T,\ \ i,j,k=1,2,\ldots,\hn$ symmetric in $i,j,k$ such that
\begin{equation}
\label{Someday}
\gm_{\Ra_k}=\left(
\begin{array}{*{4}{c}}
0 & \delta^j_k& 0 & 0 \\
0 & 0 & \Yuk_{kij}& 0 \\
0 & 0 & 0 & \delta_k^i \\
0 & 0 & 0 & 0
\end{array} \right), \ \ \ \ \ \ \  
\end{equation} 
Further
\begin{equation}
\label{santana}
\Ra_{i_1}\Yuk_{i_2i_3i_4}=\Ra_{i_2}\Yuk_{i_1i_3i_4}.\ \ \ 
\end{equation}
\item
For any $\gg\in \Lie(\BG)$ there is also a unique vector field $\Ra_{\gg}$  in $\T$ 
such that
\begin{equation}
\gm_{\Ra_\gg}=\gg^{\tr}.
\end{equation}
\end{enumerate}
\end{theo}

Our proof of Theorem \ref{maintheo} is based on techniques from special geometry which deals with periods of Calabi-Yau varieties,
see for instance Refs.~\cite{Candelas:1990pi,Candelas:1990rm,Strominger:1990pd,Ceresole:review,Alim:2012gq} and  \S\ref{specialgeometry}. 
For particular examples, such as the mirror quintic, one can give an algebraic proof which is merely computational, 
see \S\ref{mqc}. Further partial results in this direction are obtained in Ref.~\cite{Nikdelan14}. The second part of Theorem \ref{maintheo}
can be proved using algebraic methods and this will be discussed in subsequent works. 

The $\O_\T$-module $\liealg$ generated by the 
vector fields 
\begin{equation}
\label{29apr2014}
\Ra_i,\ \   \Ra_{\gg},\ \ i=1,2,\ldots,\hn,\ \ 
,\ \ \gg\in \Lie(\BG)\,,
\end{equation}
 form a
$\dim(\T)$-dimensional Lie algebra with the usual bracket of vector fields.
In the case of enhanced moduli of elliptic curves, one gets the classical Lie algebra 
${\mathfrak s}{\mathfrak l}_2$, see for instance Refs.~\cite{ho07-1, Guillot07}.   
Our main motivation for introducing such vector fields is that they are basic ingredients for an algebraic version of
the Bershadsky-Cecotti-Ooguri-Vafa
holomorphic anomaly equations \cite{Bershadsky:1993ta,Bershadsky:1993cx}. First, we choose a basis of $\Lie(\BG)$:
\begin{equation}
\label{gofLie}
\tgtwo_{ab}:=\left( 
\begin{array}{cccc}
 0 & 0 & 0 & 0 \\
 0 & 0 & 0 & 0 \\
 0& \frac{1}{2} (\delta^{i}_{a} \delta^{j}_{b} + \delta^{i}_{b} \delta^{j}_{a} ) & 0 & 0 \\
 0 & 0 & 0 & 0 \\
\end{array}
\right) \
\tgone_{a}=
\left(
\begin{array}{cccc}
 0 & 0 & 0 & 0 \\
 0 & 0 & 0 & 0 \\
 -\delta^{i}_{a} & 0 & 0 & 0 \\
 0 & \delta_{a}^{j} & 0 & 0 \\
\end{array}
\right)\,
\tgzero:=\left(
\begin{array}{cccc}
 0 & 0 & 0 & 0 \\
 0 & 0 & 0 & 0 \\
 0 & 0 & 0 & 0 \\
 -1 & 0 & 0 & 0 \\
\end{array}
\right)\,
\end{equation}
$$
\kgone_{a}:= \left(
\begin{array}{cccc}
 0 & 0 & 0 & 0 \\
 \delta^{a}_{i} & 0 & 0 & 0 \\
 0 & 0 & 0 & 0 \\
 0 & 0 & \delta^{a}_{j} & 0 \\
\end{array}
\right)\,,  \ \
\ggtwo^{a}_{b} :=
 \left(
\begin{array}{cccc}
 0 & 0 & 0 & 0 \\
 0 & -\delta_{i}^{a} \delta^{j}_{b}   & 0 & 0 \\
 0 & 0 & \delta_{b}^{i} \delta_{j}^{a} & 0 \\
 0 & 0 & 0& 0 \\
\end{array}\right)\,
\ggzero_0:=
\left(
\begin{array}{cccc}
 -1 & 0 & 0 & 0 \\
 0 & 0 & 0 & 0 \\
 0 & 0 & 0 & 0 \\
 0 & 0 & 0& 1 \\
\end{array}\right)\,.
$$
and we call it the canonical basis.  
The Lie algebra structure of $\liealg$ is given by the following table.

{\tiny
\begin{equation}
\label{liebrackettable}
\begin{array}{|c|c|c|c|c|c|c|c|}
\hline
&\Ra_{\ggzero_0} & \Ra_{\ggtwo_{c}^{d}}&\Ra_{\tgtwo_{cd}} & \Ra_{\tgone_{c}} &  \Ra_{\tgzero} &\Ra_{\kgone^{c}}&\Ra_{c} \\

\hline
 \Ra_{\ggzero_0} &0&0&0&-\Ra_{\tgone_{c}}&-2\Ra_{\tgzero}&-\Ra_{\kgone^{c}} &\Ra_{c}\\

\hline
\Ra_{\ggtwo_{b}^{a}}&0&0&-\delta_{c}^{a} \Ra_{\tgtwo_{bd}}- \delta_{d}^{a} \Ra_{\tgtwo_{bc}}&- \delta^{a}_{c} \Ra_{\tgone_{b}}&0&\delta^{c}_{b} \Ra_{\kgone^{a}} &-\delta^{a}_{c} \Ra_{b}\\

\hline
\Ra_{\tgtwo_{ab}} &0&\delta_{a}^{d} \Ra_{\tgtwo_{bc}}+\delta_{b}^{d} \Ra_{\tgtwo_{ac}}&0&0&0&\frac{1}{2} (\delta_{a}^{c} \Ra_{\tgone_{b}} + \delta_{b}^{c} \Ra_{\tgone_{a}} )&-\frac{1}{2} ( \Yuk_{cbd} \Ra_{\ggtwo_{a}^{d}}  +\Yuk_{acd} \Ra_{\ggtwo_{b}^{d}}    )\\

\hline 
\Ra_{\tgone_{a}} &\Ra_{\tgone_{a}} &\delta^{d}_{a} \Ra_{\tgone_{c}}&0&0&0&2 \delta^{c}_{a}  \Ra_{\tgzero}&2 \Ra_{\tgtwo_{ac}} -\Yuk_{acd} \Ra_{\kgone^{d}} \\

\hline 
\Ra_{\tgzero} &2 \Ra_{\tgzero} &0&0&0&0&0&\Ra_{\tgone_{c}}\\

\hline
\Ra_{\kgone^{a}} &\Ra_{\kgone^{a}}&-\delta^{a}_{c} \Ra_{\kgone^{d}}&-\frac{1}{2} (\delta_{c}^{a} \Ra_{\tgone_{d}} + \delta_{d}^{a} \Ra_{\tgone_{c}} )& -2\delta^{a}_{c} \Ra_{\tgzero}&0&0&-\delta_{c}^{a} \Ra_{\ggzero_0} +\Ra_{\ggtwo_{c}^{a}}\\

\hline
\Ra_{a}&-\Ra_{a}&\delta_{a}^{d} \Ra_{c}&
\frac{1}{2} ( \Yuk_{ade} \Ra_{\ggtwo_{c}^{e}}  +\Yuk_{ace} \Ra_{\ggtwo_{d}^{e}}    )&
-2 \Ra_{\tgtwo_{ac}} +\Yuk_{ace} \Ra_{\kgone^{e}} &
-\Ra_{\tgone_{a}}&
\delta_{a}^{c} \Ra_{\ggzero_0} -\Ra_{\ggtwo^{c}_{a}}&
0\\
\hline
\end{array}
\end{equation}
}
The genus one topological string partition function $\Fg_1^\alg$ belong to $\log(\O_\T^*)$, 
where $\O_\T^*$ is the set of invertible regular functions in $\T$,  
and it satisfies the following equations:
\begin{eqnarray}
\label{aaeq1}
\Ra_{\ggzero_0}  \Fg_1^\alg&=& -\frac{1}{2} (3+\frac{\chi}{12}) \,,\\
\Ra_{\ggtwo_{{b}}^{a}}\, \Fg_1^\alg &=& -\frac{1}{2} \delta^{b}_{{a}}\,, \\
\Ra_{\gg}\, \Fg_1^\alg &=&0,\ \ \ \text{ all other $\gg$ of the canonical basis of } \Lie(\BG). 
\end{eqnarray}
Here, $\chi$ is the Euler number of the Calabi-Yau variety $X$. 
The genus $g$ topological string partition function $\Fg_g^\alg\in \O_\T$ turns out to be 
a regular function in $\T$.
The holomorphic anomaly equations in the polynomial formulation of Refs.~\cite{Yamaguchi:2004bt, Alim:2007qj} can be written in terms of vector fields:
\begin{eqnarray}
\label{aaeq}
\Ra_{\tgtwo^{{a}{b}} }{\Fg}_{g}^\alg &=& \frac{1}{2} \sum_{h=1}^{g-1} \Ra_{{a}} {\Fg}_{h}^\alg\, \Ra_{{b}} {\Fg}_{g-h}^\alg + \frac{1}{2} \Ra_{{a}} \Ra_{{b}} {\Fg}_{g-1}^\alg\,, \\ \nonumber
\Ra_{\kgone_{{a}}} {\Fg}_g^\alg &=&0\,,\\ \nonumber
\Ra_{\ggzero_0}  {\Fg}_g^\alg&=&(2g-2) {\Fg}_g^\alg\,, \\ \nonumber
\Ra_{\ggtwo_{{a}}^{{b}}} {\Fg}_g ^\alg&=& 0\,.
\end{eqnarray}
The functions $\Fg_g^\alg$ are not defined uniquely by the algebraic holomorphic anomaly equation as above. 
Let $\Ts$ be the moduli of $(X,\omega_1)$, where $X$ is a Calabi-Yau threefold as above and $\omega_1$ is a holomorphic differential $3$-form on $X$.  We have a canonical projection $\T\to \Ts$ which is obtained by neglecting all $\omega_i$'s except $\omega_1$. It is characterized by the fact that $f\in\O_\Ts$ does not depend on the choice of $\omega_2,\cdots,\omega_{2\hn+2}$. 
We also 
expect  that $\Ts$ has a canonical structure of an affine variety over $\bar\Q$ such that $\T\to \Ts$ is a morphism of 
affine varieties. 
We get a sub-algebra $\O_\Ts$ of $\O_\T$ which is characterized by the following:
\begin{theo}
\label{theo2}
We have 
\begin{equation}
\label{40salnazdik}
\bigcap_{\gg\in \text{ canonical basis }  \gg\not ={\ggzero_0}} \ker(\Ra_{\gg})
=\O_\Ts\,,
\end{equation}
where we regard a vector field in $\T$ as a derivation in $\O_\T$. 
\end{theo}
This means that $\Fg_g^\alg,\ \ g\geq 2$ (resp. $\Fg_1^\alg$) is defined up to addition of an element of $\O_\Ts$ (resp. $\log (\O_\Ts^*)$), 
which is called the ambiguity of $\Fg_g^\alg$. The algebra $\O_\T$ can be considered as a generalization of the classical algebra
of quasi-modular forms. For a discussion of the $q$-expansion of its elements see Refs.~\cite{Yamaguchi:2004bt,Alim:2007qj, ho21}. The results are based on considerable amount of machine and hand computations which are suppressed in this paper to enhance readability.

The text is organized in the following way.
In \S\ref{physicssection} we review basic facts about special geometry, the original BCOV holomorphic anomaly equation, the polynomial structure of
topological string partition functions. New manipulations of periods inspired by our geometric approach are explained in \S\ref{16sep2014} and \S\ref{16sep201-1}.
\S\ref{proofssection} is dedicated to the proofs of our main theorems.
In \S\ref{GPDsection} we first recall the definition of a generalized period domain for Calabi-Yau threefolds.  Via the generalized period maps, we interpret the polynomial generators
and topological string partition functions as functions on the moduli space $\T$. 
The mathematical content of the period manipulations of special geometry are explained in \S\ref{twoequalitiessection}.  Explicit computations of the
vector fields \eqref{29apr2014} and the construction of the moduli space $\T$ in the case of mirror quintic is explained in \S\ref{mqc}. Finally, in \S\ref{remarkssection}
we review some works for future and possible applications of our algebraic anomaly equation. 

\subsection*{Acknowledgements}
We would like to thank Jie Zhou for discussions and for collaborations on related projects. H.M. would like to acknowledge the generous support he received from Brazilian science without border program in order to stay for a sabbatical year at Harvard University. E.S. would like to thank the mathematics department at Harvard University for support and hospitality. This work has been supported by NSF grants PHY-1306313 and DMS-0804454.

\section{Holomorphic anomaly equations}
\label{physicssection}
In this section we review some basic formulas used in special geometry of Calabi-Yau threefolds.
We use physics conventions of writing out the components of geometric objects and for handling indices. In general (lower) upper indices will denote components of (co-) tangent space. Identical lower and upper indices are summed over, i.~e.~$x^iy_i := \sum_i x^i y_i$.  For derivatives w.r.t. coordinates $x^i$ we will write $\partial_{i}:= \frac{\partial}{\partial x^i}$ and 
$\partial_{\bar\imath}:= \frac{\partial}{\partial \bar x^{\ib}}$. The inverse of a matrix $[M_{ij}]$ is denoted by $[M^{ij}]$. We define  $\delta^i_j$ to be $1$ if $i=j$ and $0$ otherwise. 

\subsection{Special geometry}
\label{specialgeometry}
By Bogomolov-Tian-Todorov the moduli space $\mathcal{M}$ of projective Calabi-Yau threefolds is smooth and hence we can take local coordinates 
$z=(z^1,z^2,\ldots,z^\hn)\in (\C^\hn,0)$ for an open set $U$ in such a moduli space. In our context, $\mathcal{M}$ is just the quotient of $\T$ by
the action of the algebraic group $\BG$.
We denote by $\Omega=\Omega_z$
a holomorphic family of $3$-forms on the Calabi-Yau threefold $X_z$.
The geometry of $\mathcal{M}$ can be best described using the third cohomology bundle 
$\mathcal{H}\rightarrow \mathcal{M}$, where the fiber of $\mathcal H$
at  a point $z\in\mathcal M$ is $\mathcal{H}_z=H^3(X_z,\C)$.  
This bundle can be decomposed into sub-bundles in the following way:
\begin{equation}
\mathcal{H}= \mathcal{L} \oplus \left(\mathcal{L}\otimes \,T\mathcal{M}\right)  \oplus \overline{\left(\mathcal{L}\otimes \,T\mathcal{M}\right)}  \oplus \overline{\mathcal{L}}\,,
\end{equation}
where $\mathcal{L}$ is the line bundle  of holomorphic $(3,0)$ forms in $X_z$, 
$T\mathcal{M}$ denotes the holomorphic tangent bundle and the overline denotes complex conjugation. 
This structure gives the Hodge decomposition of the variation of 
Hodge structure arising from the Calabi-Yau threefolds $X_z$. It carries the intersection form
in cohomology 
$$
\langle\cdot,\cdot \rangle: \mathcal{H}\times \mathcal{H}\to\C \,,
$$
which is given by the same formula as in \eqref{16s2014}. Let
$$
\partial_i=\partial/\partial z^i\,, \partial_{\bar{\jmath}}=\partial/\partial \bar{z}^{\jmath}
$$
These are sections of $T\mathcal{M}$ and $\Omega$ is a section of $\mathcal L$.
We get in canonical way global sections of the (symmetric) tensor product of $\mathcal L$, $T\mathcal{M}$
and $T^*\mathcal{M}$ which form a basis at each fiber. Any  other section can be written
as a linear combination of such a basis and we treat such coefficients as if they are sections themselves.
Let
$$
K:=-\log \langle \Omega, \overline{\Omega}\rangle\,,
$$
be the K\"ahler potential. It provides a K\"ahler form for a K\"ahler metric on $\mathcal{M}$, whose components and Levi-Civita connection are given by:
\begin{equation}
G_{i\bar{\jmath}} := \partial_i \partial_{\bar{\jmath}} K\,, \quad \Gamma_{ij}^k= G^{k\bar{k}} \partial_i G_{j\bar{k}}\,.
\end{equation}
The description of the change of the decomposition of $\mathcal{H}$ into sub-bundles is captured by the holomorphic Yukawa couplings or threepoint functions
\begin{equation}\label{Yuk}
C_{ijk}:=-\langle \Omega, \partial_{i}\partial_{j}\partial_{k}\Omega \rangle \in \Gamma\left( \mathcal{L}^2 \otimes \textrm{Sym}^3 T^*\mathcal{M}\right)\,,
\end{equation}
which satisfy
\begin{equation}\label{Yukeq}
\partial_{\bar{\imath}} C_{ijk}=0\,, \quad D_i C_{jkl}= D_j C_{ikl},
\end{equation}
the curvature is then expressed as \cite{Bershadsky:1993cx}:
 \begin{equation}
 [\bar{\partial}_{\bar{\imath}},D_i]^l_{\phantom{l}j}=\bar{\partial}_{\bar{\imath}} \Gamma^l_{ij}= \delta_i^l
G_{j\bar{\imath}} + \delta_j^l G_{i\bar{\imath}} - C_{ijk} \overline{C}^{kl}_{\bar{\imath}},
\label{curvature}
 \end{equation}
where
\begin{equation}
\overline{C}_{\bar{\imath}}^{jk}:= e^{2K} G^{k\bar{k}} G^{j\bar{\jmath}}\overline{C}_{\bar{\imath}\bar{k}\bar{\jmath}}\,,
\end{equation}
and where $D_i$ is the covariant derivative defined using the connections 
$\Gamma_{ij}^k$ and $K_i$, for example for $A_j^k$ a section of 
$\mathcal{L}^{n} \otimes T^*\mathcal{M} \otimes T\mathcal{M}$ we have
$$
D_i A_j^k:= \partial_i\, A_j^k -\Gamma_{ij}^m A_m^k + \Gamma_{im}^k A^m_j +  n\, K_i A_j^k\,.
$$
We further introduce the objects $S^{ij},S^i,S$, which are sections of $\mathcal{L}^{-2}\otimes \text{Sym}^m T\mathcal{M}$ with $m=2,1,0$, respectively, and give local potentials for the non-holomorphic Yukawa couplings:
\begin{equation}
\partial_{\bar{\imath}} S^{ij}= \overline{C}_{\bar{\imath}}^{ij}, \qquad
\partial_{\bar{\imath}} S^j = G_{i\bar{\imath}} S^{ij}, \qquad
\partial_{\bar{\imath}} S = G_{i \bar{\imath}} S^i.
\label{prop}
\end{equation}

%


\subsection{Special coordinates}
\label{specialcoordinates}
We pick a symplectic basis $\{ A^I, B_J\}\,, I,J=0,\dots,\hn$ of $H_{3}(X_z,\Z)$, satisfying
\begin{equation}\label{sympbasis}
A^I\cdot B_J=\delta_I^J\,,\quad A^I\cdot A^J= 0\,, \quad B_I\cdot B_J=0 \,,
\end{equation}
We write the periods of the holomorphic $(3,0)$ form over this basis:
\begin{equation}
X^I(z):=\int_ {A^I}\Omega_z,\ \  \mathcal{F}_J(z):=\int_{B_J}\Omega_z\, .
\end{equation}
The periods $X^I(z),\mathcal{F}_J(z)$ satisfy the Picard-Fuchs equations of the Calabi-Yau family $X_z$. 
The periods $X^I$ can be identified with projective coordinates on $\mathcal{M}$ and 
$\mathcal{F}_J$ with derivatives of a homogeneous function $\mathcal{F}(X^I)$ of weight 
2 such that $\mathcal{F}_J=\frac{\partial \mathcal{F}(X^I)}{\partial X^J}$. 
In a patch where $X^0(z)\ne 0$ a set of special coordinates can be defined
\begin{equation} \label{special}
t^a=\frac{X^a}{X^0}\, ,\quad a=1,\dots ,\hn.
\end{equation}
The normalized holomorphic $(3,0)$ form $ \tilde{\Omega}_t :=(X^0)^{-1} \Omega_z$ has the periods:
\begin{equation}
\label{cmsa2014}
\int_{ A^0, A^a, B_b, B_0}\tilde \Omega_t= \left (1,t^a, F_b(t), 2F_0(t)-t^c F_c(t) \right) \,,
\end{equation}
where $$F_0(t)= (X^0)^{-2} \mathcal{F} \quad \textrm{and} \quad F_a(t):=\partial_a F_0(t)=\frac{\partial F_0(t)}{\partial t^a}.$$
$F_0(t)$ is the called the prepotential and 
\begin{equation}
C_{abc}=\partial_a \partial_b \partial_c F_0(t)\,.
\end{equation}
are the Yukawa coupling in the special coordinates $t^a$. See Ref.~\cite{Ceresole:review,Alim:2012gq} for more details.


\subsection{Holomorphic anomaly equations}
\label{BCOVHAE}
The genus $g$ topological string amplitude $\Fg^\non_g$  are defined in Ref.~\cite{Bershadsky:1993ta} 
for $g=1$ and Ref.~\cite{Bershadsky:1993cx} for $g\geq 2$.  It is a section of the line bundle
$\mathcal{L}^{2-2g}$ over $\mathcal M$.
They are related recursively in $g$ by 
the holomorphic anomaly equations \cite{Bershadsky:1993ta}
\begin{equation}
\bar{\partial}_{\bar{\imath}} \partial_j\Fg^{\non}_1 = \frac{1}{2} C_{jkl}
\overline{C}^{kl}_{\bar{\imath}}+ (1+\frac{\chi}{24})
G_{j \bar{\imath}}\,, \label{anom2}
\end{equation}
where $\chi$ is the Euler character of B-model CY threefold, and \cite{Bershadsky:1993cx}
\begin{equation}
\bar{\partial}_{\bar{\imath}} \Fg^{\non}_g = \frac{1}{2} \overline{C}_{\bar{\imath}}^{jk} \left(
\sum_{r=1}^{g-1}
D_j\Fg^\non_{r} D_k\Fg^\non_{(g-r)} +
D_jD_k\Fg_{g-1}^{\non} \right) \label{anom1}.
\end{equation}
Note that $D_i\Fg_g^{\non}$ is a section of ${\mathcal L}^{2-2g}\otimes T^*{\mathcal M} $.

\subsection{Polynomial structure}
\label{17september}
In Ref.~\cite{Yamaguchi:2004bt} it was shown that the topological string amplitudes for the mirror quintic can be expressed as polynomials in 
finitely many generators of differential ring of multi-derivatives of the connections of special geometry. 
This construction was generalized in Ref.~\cite{Alim:2007qj} for any CY threefold. It was shown there that 
$\Fg_{g}^\non$ is a polynomial of degree $3g-3$ in the generators $S^{ij},S^i,S,K_i$, 
where degrees $1,2,3,1$ were assigned to these generators respectively. The proof was given inductively and relies on the closure of these generators under the holomorphic derivative \cite{Alim:2007qj}. The purely holomorphic part of the construction as well as the coefficients of the monomials would be rational functions in the algebraic moduli, this was further discussed in Refs.~\cite{Alim:2008kp,Hosono:2008np}.

It was further shown in Ref.~\cite{Alim:2007qj}, following Ref.~\cite{Yamaguchi:2004bt}, that the L.H.S. of Eq.~(\ref{anom1}) could 
be written in terms of the generators using the chain rule:
\begin{equation}
\label{25sep2014}
\partial_{\ib} \Fg_g^{\non} = \overline{C}_{\ib}^{jk} 
\frac{\partial \Fg_g^{\non}}{\partial S^{jk}} + G_{i\ib} \left( S^{ij} \frac{\partial \Fg_g^{\non}}{\partial S^{j}} + 
S^i \frac{\partial \Fg_g^{\non}}{\partial S} +\frac{\partial \Fg_g^{\non}}{\partial K_i}\right) \,,
\end{equation}
assuming the linear independence of $\overline{C}_{\ib}^{jk}$ and $G_{i\ib}$ 
over the field generated by the generators in Definition \ref{specgen},
the holomorphic anomaly equations (\ref{anom1}) could then be written as two different sets of equations \cite{Alim:2007qj}:
\begin{eqnarray}\label{anompol1}
\frac{\partial \Fg^{\non}_g}{\partial S^{jk}} = \frac{1}{2}
\sum_{r=1}^{g-1}
D_j\Fg^\non_{r} D_k\Fg^\non_{(g-r)} +\frac{1}{2}
D_jD_k\Fg^\non_{(g-1)}  , \\ \label{anompol2}
 S^{ij} \frac{\partial \Fg_g^{\non}}{\partial S^{j}} + S^i \frac{\partial \Fg_g^{\non}}{\partial S} +\frac{\partial \Fg_g^{\non}}{\partial K_i}=0\,. 
\end{eqnarray}
This linear independence assumption is not at all a trivial statement. Its proof in the one parameter case can be done using a differential Galois  theory argument as in the proof of Theorem 2 in \cite{ho21}  and one might try to generalize such an argument to multi parameter case. However, for the purpose of the present
text we do not need to prove it. The reason is that \eqref{anompol1} and \eqref{anompol2} always have a solution which gives a solution to (\ref{25sep2014}).


\section{Algebraic structure of topological string theory}
In this section we develop the new ingredients and tools which will allow us to phrase the algebraic structure of topological string theory. We start by enhancing the differential polynomial ring of Ref.~\cite{Alim:2007qj} with further generators which parameterize a choice of section of the line bundle $\mathcal{L}$ and a choice of coordinates as was done in Ref.~\cite{Alim:2013eja} for one dimensional moduli spaces. We will then show that these new generators parameterize different choices of forms compatible with the Hodge filtration and having constant symplectic pairing. 

\subsection{Special polynomial rings}
We first fix the notion of holomorphic limit discussed in Ref.~\cite{Bershadsky:1993cx}. For our purposes we think of the limit as an assignment:
\begin{equation}
e^{-K}|_{\textrm{hol}} = \mathsf{h}_0 \, X^0 \,, \quad G_{i\jb}|_{\textrm{hol}}= \mathsf{h}_{a\jb} \,\frac{\partial t^a}{\partial z^i}
\end{equation}
for a given choice of section $X^0$ of $\mathcal{L}$ and a choice of special coordinates $t^a$\, where $\mathsf{h}_0$ is a constant and $\mathsf{h}_{a\ib}$ denote the components of a constant matrix.

\begin{defi}\label{specgen}
The generators of the special polynomial differential ring are defined by
\begin{align}
g_0 &:=\mathsf{h}_0^{-1} e^{-K} \,, \\
 g^a_i& :=e^{-K} G_{i\jb}\, \mathsf{h}^{\jb a} \,,\\
T^{ab} & := g^a_i \, g^b_j S^{ij}  \,,\\
T^a & :=  g_0 \, g^a_i (S^i-S^{ij}K_j)\,,\\
T & := g_0^2 (S-S^i K_i + \frac{1}{2} S^{ij} K_i K_j)\,,\\
L_a &= g_0 (g^{-1})^{i}_a \partial_i K\,.
\end{align}
We will use the same notation for these generators and for their holomorphic limit. 
\end{defi}

\begin{prop}
\label{28s2014}
The generators of the special polynomial ring satisfy the following differential equations, called the differential ring:
\begin{align}
 \partial_a g_0 &= -L_a \, g_0 \,, \\ 
\partial_a g_i^b &= g_i^c\left( \delta^{b}_a\,L_c - C_{cad} T^{bd} + g_0\, s_{ca}^b\right)\,,\\
\partial_a T^{bc} & = \delta_{a}^{b} (T^c + T^{cd} L_d)+  \delta_{a}^{c} (T^b + T^{bd} L_d) - C_{ade} T^{bd} T^{ce} + g_0\, h_{a}^{bc}\, ,\\
\partial_a T^b &=  2 \delta^{b}_a (T + T^c L_c) -T^b\, L_a - k_{ac} T^{bc} + g_0^2\, h^b_a\, ,\\
\partial_a T & =  \frac{1}{2} C_{abc} T^b T^c - 2 L_a T- k_{ab} T^b + g_0^3\, h_a\, , \\
\partial_a L_b & =  -L_a L_b -C_{abc} T^c +  g_0^{-2}\, k_{ab}\,.
\end{align}
\end{prop}

\begin{proof}
The first two equations follow from  Definition \ref{specgen} and the special geometry discussed in Sec.~\ref{specialgeometry}, the other equations follow from the definitions and the equations which were proved in Ref.~\cite{Alim:2007qj}.
\end{proof}
The generators $g_0,g_i^a$ are chosen such that their holomorphic limits become:
\begin{equation}
g_0|_{\textrm{hol}} = X^0\,,\quad g_{i}^a|_{\textrm{hol}}= X^0\,\frac{\partial t^a}{\partial z^i}\,, 
\end{equation}
In these equations the functions $s_{ij}^k, h_i^{jk},h_i^j,h_i$ and $k_{ij}$ are fixed once a choice of generators has been made and we transformed the indices from arbitrary algebraic coordinates to the special coordinates using $g_i^a$ and its inverse.

The freedom in choosing the generators $S^{ij},S^i,S$ was discussed in Ref.~\cite{Alim:2008kp,Hosono:2008np} and translates here to a freedom of adding holomorphic sections $\mathcal{E}^{ij},\mathcal{E}^i,\mathcal{E}$ of $\mathcal{L}^{-2}\otimes \text{Sym}^m T\mathcal{M}$ with $m=2,1,0$, respectively to the generators as follows:
\begin{eqnarray}
T^{ab} &\rightarrow& T^{ab} + g^{a}_i \, g_j^b\, \mathcal{E}^{ij}\,,\\
T^a &\rightarrow& T^a + g_0\, g^a_i \mathcal{E}^i \,,\\
T &\rightarrow& T + g_0^2 \mathcal{E}\,.
\end{eqnarray}
It can be seen from the equations that there is additional freedom in defining the generators $g_0,g_i^a$ and $L_a$ given by:
\begin{eqnarray}
L_a &\rightarrow& L_{a} + g_0 (g^{-1})_{a}^i  \mathcal{E}_i\,,\\
g_0 &\rightarrow & \mathcal{C} g_0\, ,\\
g_i^a & \rightarrow & \mathcal{C}^i_j g_i^a\,,
\end{eqnarray}
where $\mathcal{C}$ denotes a holomorphic function, $\mathcal{C}_{j}^i$ a holomorphic section of $T\mathcal{M} \otimes T^*\mathcal{M}$ and $\mathcal{E}_i$ a holomorphic section of $T^*\mathcal{M}$.

The number of special polynomial generators matches $\dim(\BG)$, where $\BG$
is the algebraic group in the Introduction. 


\begin{defi}
We introduce:
\begin{equation}
\tilde\Fg^{\non}_g := g_0^{2g-2} \Fg_g^{\non}\,,
\end{equation}
which defines a section of $\overline{\mathcal{L}}^{2g-2}$. 
After taking the holomorphic limit discussed earlier we get $\Fg_g^\hol$ 
which will be a holomorphic function (and no longer a section) on the moduli space $\mathcal{M}$.
\end{defi}

\begin{prop}\label{anomsplit}
$\tilde\Fg^{\non}_g$'s satisfy the following equations:

\begin{eqnarray}
\left( g_0 \frac{\partial}{\partial g_0}   + L_a \frac{\partial}{\partial L_a}  + T^a \frac{\partial}{\partial T^a} + 2 T\frac{\partial}{\partial T}\right) \tilde\Fg_{g}^{\non} = (2g-2) \tilde\Fg_{g}^{\non} \, \\ 
\left( g^a_m \frac{\partial}{\partial g_m^b} + 2 T^{ac} \frac{\partial}{\partial T^{bc}} + T^a \frac{\partial}{\partial T^b} - L_b \frac{\partial}{\partial L_a} \right) \tilde\Fg_{g}^{\non} = 0\,, \\
\left(\frac{\partial}{\partial T^{ab}} -\frac{1}{2}(L_b \frac{\partial}{\partial T^a}+ L_a \frac{\partial}{\partial T^b}) +\frac{1}{2} L_a L_b \frac{\partial}{\partial T} \right) \tilde\Fg^{\non}_g \nonumber\\=\frac{1}{2} \sum_{r=1}^{g-1} \partial_a \tilde\Fg^{\non}_r\, \partial_b \tilde\Fg^{\non}_{g-r} + \frac{1}{2} \partial_a \partial_b \tilde\Fg^{\non}_{g-1} \,,\\
\frac{\partial \tilde\Fg^{\non}_g}{\partial L_a} = 0\,.
\end{eqnarray}
$\Fg_g^\hol$'s satisfy the same equations in the holomorphic limit. 
\end{prop}
\begin{proof}
The first two equations follow from the definition of $\tilde{\Fg}_{g}^{\non}$ and the proof of Ref.~\cite{Alim:2007qj}, bearing in mind that the dependence on the generators $g_0,g_i^a$ is introduced through the definition of the special polynomial generators and the factor $g_0^{2g-2}$ in $\tilde{\Fg}_{g}^{\non}$. The third and fourth equation are a re-writing of Eqs.~(\ref{anompol1}) using the special polynomial generators defined earlier.
\end{proof}


\subsection{Different choices of Hodge filtrations}
In order to parameterize the moduli space of a Calabi-Yau threefold enhanced with a choice of forms compatible with the Hodge filtration and having constant intersection, we seek to parameterize the relation between different choices of Hodge filtrations. We start with a choice of a Hodge filtration $\vec{\omega}_z$ defined by the algebraic coordinates on the moduli space and relate this to a choice of filtration in special coordinates $\vec{\omega}_t$. 
The choices are given by
\begin{equation}
\vec{\omega}_z = \left(\begin{array}{c} \alpha_{z,0} \\ \alpha_{z,i} \\ \beta_z^i \\ \beta_z^0\end{array} \right)\left( \begin{array}{c} \Omega  \\ \partial_i \Omega \\ (C_{\sharp}^{-1})^{ik} \partial_{\sharp} \partial_k \Omega \\ \partial_{\sharp} (C_{\sharp}^{-1})^{\sharp k} \partial_* \partial_k \Omega \end{array}\right),
\end{equation}
where $C_{ijk}$ are given by \eqref{Yuk}. Here, $A_{\sharp}= (g^{-1})^{i}_* A_i$, where $*$ denotes a fixed choice of special coordinate.
We also define
\begin{equation}
\label{6oct2014}
\vec{\omega}_t= \left(\begin{array}{c} \alpha_{t,0} \\ \alpha_{t,a} \\ \beta_t^a \\ \beta_t^0\end{array} \right)=
\left(\begin{array}{c}
\tilde{\Omega}\\
\partial_a \tilde{\Omega} \\
(C_*^{-1})^{ae}\partial_* \partial_e \tilde{\Omega}\\
\partial_* (C_*^{-1})^{*e}\partial_* \partial_e \tilde{\Omega}
\end{array}
\right)\,,
\end{equation}
where $\tilde\Omega=\tilde\Omega_t$ is given by \eqref{cmsa2014} and
$*$ denotes a fixed choice of special coordinate.

\begin{prop}
\label{18sep2014}
The period matrix of $\vec{\omega}_t$ over
the symplectic basis of $H_3(X_z,\Z)$ given in Eq.~(\ref{sympbasis}) has the following special format:
\begin{equation}
\label{specialmatrix}
[\int_{{A^0,A^c,B_c,B_0} }\left(\begin{array}{c} \alpha_0 \\ \alpha_a \\ \beta^a \\ \beta^0\end{array} \right)]= \left(\begin{array}{cccc}
1 & t^c & F_c & 2F_0-t^dF_d\\
0&\delta^c_a & F_{ac}& F_a-t^d F_{ad}\\
0&0&\delta^a_c & -t^a \\
0&0&0&-1
\end{array}
\right)\,,
\end{equation}
with $F_a:=\partial_a F_0,\quad \partial_a=\frac{\partial}{\partial t^a}.$
\end{prop}
\begin{proof}
 This follows from the defitions of Sec.~\ref{specialcoordinates}
\end{proof}

\begin{prop}
\label{myfav}
The symplectic form for both bases $\vec{\omega}_z$ and $\vec{\omega}_{t}$ is the matrix $\imc$ in (\ref{31aug10}).
\end{prop}

\begin{proof}
The computation of the symplectic form for 
$\vec{\omega}_t$ follows from the Proposition \ref{18sep2014}. 
The symplectic form of $\vec{\omega}_z$ follows from the definition of $C_{ijk}$  
in (\ref{Yuk}) and from Griffiths transversality, for instance,
\begin{equation}
\int_{X_z} \partial_i \Omega \wedge (C^{-1}_{\sharp})^{jk} \partial_{\sharp} \partial_k \Omega= - (C_{\sharp}^{-1})^{jk} \int_{X_z} \Omega \wedge \partial_i \partial_{\sharp} \partial_k \Omega = (C_{\sharp}^{-1})^{jk} C_{\sharp ik} = \delta_i^j\,.
\end{equation}
\end{proof}

\begin{prop}
The flat choice $ \vec{\omega}_{t} $ satisfies the following equation:

\begin{equation}
\label{determinantal}
\partial_b \left(\begin{array}{c}
\tilde{\Omega}\\
\partial_a \tilde{\Omega} \\
(C_*^{-1})^{ae}\partial_* \partial_e \tilde{\Omega}\\
\partial_* (C_*^{-1})^{*e}\partial_* \partial_e \tilde{\Omega}
\end{array}
\right) =   \left(\begin{array}{cccc}
0& \delta_b^c& 0& 0\\
0&0 & C_{abc}&0\\
0&0&0& \delta^a_b \\
0&0&0&0
\end{array}
\right) \left(\begin{array}{c}
\tilde{\Omega}\\
\partial_c \tilde{\Omega} \\
(C_*^{-1})^{ce}\partial_* \partial_e \tilde{\Omega}\\
\partial_* (C_*^{-1})^{*e}\partial_* \partial_e \tilde{\Omega}
\end{array}
\right)\,.
\end{equation}

\end{prop}

\begin{proof}
 This follows from Eq.~(\ref{specialmatrix})
 \end{proof}

\label{16sep2014}
\def\Am{{\sf B}}
We now want to find the matrix relating:
\begin{equation}
\label{30may04}
 \vec{\omega}_{t}  = \Am\cdot \vec{\omega_z}\,,
\end{equation}
and express its entries in terms of the polynomial generators.
The matrix $\Am$ is given by:
\begin{equation}
\label{muradalim}
\Am = \left( \begin{array}{cccc} 
g_0^{-1} &0&0&0\\
g_0^{-1} L_a & (g^{-1})^i_a &0&0\\
- g_0^{-1} \widehat{T}^a & (g^{-1})^{i}_d \widehat{T}^{ad} & g_i^a &0\\
-g_0^{-1} \left( 2T +\widehat{T}^d L_d\right) + g_0 \mathcal{H} & (g^{-1})^i_d (T^d + \widehat{T}^{de}L_e) + g_0 \mathcal{H}^{i}&g_i^e L_e &g_0
\end{array}
 \right)\,,
\end{equation}
where $a$ is an index for the rows and $i$ for the columns and where:
\begin{eqnarray}
\widehat{T}^a &=& T^a -g_0\, g_i^d \mathcal{E}^m\,, \\
\widehat{T}^{ab} &=& T^{ab} - g_m^a \,g_n^b\, \mathcal{E}^{mn}\,,\\
\mathcal{H} &=& g^*_j (g^{-1})_*^i (\partial_i \mathcal{E}^j + C_{imn}\mathcal{E}^{lj}\mathcal{E}^{m}-h_i^j)\\
\mathcal{H}^{i} &=& g_m^* (g^{-1})^n_* (-\partial_n \mathcal{E}^{im} -C_{nlk}\mathcal{E}^{il} \mathcal{E}^{km} + \delta_n^i \mathcal{E}^m + h_n^{im})\\
\mathcal{E}^{ik} &=&(C_{\sharp}^{-1})^{ij} s_{\sharp j}^k\,,\\
\mathcal{E}^{i}&=& (C_{\sharp}^{-1})^{ij} k_{\sharp j}\,.
\end{eqnarray}
From Proposition \ref{myfav} and the definition of the algebraic group $\BG$
it follows that $\Am^\tr\in \BG$.
This can be also verified using the explicit expression of $\Am$ in \eqref{muradalim}.


\subsection{Lie Algebra description}

In the following, we regard all the generators as independent variables and 
compute the following matrices:
\begin{equation}
\mathsf{M}_{\partial_g}= \frac{\partial}{\partial g} \Am \cdot \Am^{-1}\,,
\end{equation}
where $g$ denotes a generator and $\partial_g := \frac{\partial}{\partial g}$. We find
\begin{eqnarray}
\label{18.sepetember.2014}
\mathsf{M}_{\partial_{T^{ab}}}&=& \left(
\begin{array}{cccc}
 0 & 0 & 0 & 0 \\
 0 & 0 & 0 & 0 \\
 -\delta^{i}_{a} L_{b} & \frac{1}{2} (\delta^{i}_{a} \delta^{j}_{b} + \delta^{i}_{b} \delta^{j}_{a} ) & 0 & 0 \\
 -L_{a} L_{b} & \delta^{j}_{b} L_{a} & 0 & 0 \\
\end{array}
\right) \,,\\ \nonumber
\mathsf{M}_{\partial_{T^{a}}} &=& \left(
\begin{array}{cccc}
 0 & 0 & 0 & 0 \\
 0 & 0 & 0 & 0 \\
 -\delta^{i}_{a} & 0 & 0 & 0 \\
 -2 L_{a} & \delta_{a}^{j} & 0 & 0 \\
\end{array}
\right)\,,\\ \nonumber
\mathsf{M}_{\partial_T}&=&\left(
\begin{array}{cccc}
 0 & 0 & 0 & 0 \\
 0 & 0 & 0 & 0 \\
 0 & 0 & 0 & 0 \\
 -2 & 0 & 0 & 0 \\
\end{array}
\right)\,,\\ \nonumber
\mathsf{M}_{\partial_{L_{a}}}&=& \left(
\begin{array}{cccc}
 0 & 0 & 0 & 0 \\
 \delta^{a}_{i} & 0 & 0 & 0 \\
 0 & 0 & 0 & 0 \\
 0 & 0 & \delta^{a}_{j} & 0 \\
\end{array}
\right)\,, \\ \nonumber
\mathsf{M}_{\partial_{g_0}} &=& \left( \begin{array}{cccc} -g_0^{-1}&0&0&0 \\
 -g_0^{-1} L_{i}&0&0&0\\
 g_0^{-1} T^{i}&0&0&0\\
 2 g_0^{-1} (2 T_0 +T^{d} L_{d})&-g_0^{-1}T^{j}&-g_0^{-1} L_{j}&g_0^{-1}
 \end{array}\right) \,.
 \end{eqnarray}

\begin{equation}
\mathsf{M}_{\partial_{g_{m}^{a}}} =   \left( \begin{array}{cccc} 0&0&0&0 \\
g^{m}_i L_{a} & -\delta^{j}_{a} g^{m}_i &0&0\\
T^{id} g^m_d L_a+\delta_a^i g^m_d (T^{d}+T^{de}L_e) & -\delta^{i}_{a}\,T^{jd} g^m_d- \delta^j_a T^{id} g^m_d &  \delta_a^i (g^{-1})^{m}_j  &0 \\
2  L_a g^m_d(T^d+T^{de} L_e)  &  -\delta^j_a g^m_d (T^d+T^{de}L_e) - g^m_d T^{dj}  L_a & (g^{-1})^m_j L_a & 0 
\end{array}\right)\,,\\
\end{equation}
 We now look for combinations of the vector fields which give constant vector fields. We find the following:
\begin{eqnarray}
\label{18.09.2014}
\tgtwo_{ab} &=&  \mathsf{M}_{T^{ab}} -\frac{1}{2} (L_a\, \mathsf{M}_{T^b} + L_b \mathsf{M}_{T^a}) +\frac{1}{2} L_a L_b \mathsf{M}_{T}   \,,\\ \nonumber
\tgone_{a}&=&  \mathsf{M}_{T^a} - L_a \mathsf{M}_{T}\,, \\ \nonumber
\tgzero &=& \frac{1}{2}\mathsf{M}_{T}\,, \\ \nonumber
\kgone^{a}&=&   \mathsf{M}_{L_a}\,,\\ \nonumber
\ggtwo_{b}^{a}  &=& g^{a}_m \mathsf{M}_{g_{m}^{b}} - L_{b} \mathsf{M}_{L_a}  + 2 T^{ad} \mathsf{M}_{T^{db}} + T^{a} \mathsf{M}_{T^b}  \,, \\ \nonumber
\ggzero_0 &=& g_0 \mathsf{M}_{g_0} + L_{a} \mathsf{M}_{L_a} +T^{a} \mathsf{M}_{T^a} + 2 \mathsf{M}_{T}\,.
\end{eqnarray}
Therefore, we get all the elements of the Lie algebra $\Lie(\BG)$ and for each $\gg\in\Lie(\BG)$ a derivation $\Ra_\gg$ in the 
generators. 
\subsection{Algebraic master anomaly equation}
\label{16sep201-1}
In genus zero we have
\begin{equation}
C_{abc}= C_{ijk} (g^{-1})^i_a (g^{-1})^j_b (g^{-1})^k_c  g_0\,,
\end{equation}
and the anomaly equations:
\begin{eqnarray}
\Ra_{\ggzero_0} C_{abc} &=& C_{abc} \,,\\ 
\Ra_{\ggtwo_{a}^{b}} C_{cde} &=& -\delta_c^b  C_{ade} -\delta_d^b  C_{cae}-\delta_e^b  C_{cda}\,,
\end{eqnarray}
and in genus one we obtain the anomaly equations (\ref{aaeq1}) and (\ref{aaeq}) in the holomorphic context.
We combine all $\Fg_g^\hol$ into the generating function:
\begin{equation}
Z= \exp \sum_{g=1}^{\infty} \lambda^{2g-2} \, \Fg^\hol_g\,.
\end{equation}
The master anomaly equations become:
\begin{eqnarray}
\Ra_{\ggzero_0} \, Z &=& \left(  -\frac{3}{2} -\frac{\chi}{24} + \theta_{\lambda} \right) Z \,,\\
\Ra_{\ggtwo^b_{a}} Z &=&-\frac{1}{2}\delta^{b}_a\, Z\,,\\
\Ra_{\tgtwo_{ab}} Z &=& \frac{\lambda^2}{2} \Ra_a \Ra_b Z\,, \\
\Ra_{\kgone_a} Z &=& 0\,,\\
\Ra_{\tgone_a} Z &=& \lambda^2 \left( -\frac{\chi}{24} + \theta_{\lambda}\right) \Ra_a Z\,, \\
\Ra_{\tgzero} Z &=& \frac{\lambda^2}{2} \left(- \frac{\chi}{24} + \theta_{\lambda} \right) \left(-\frac{\chi}{24} -1 + \theta_{\lambda} \right) Z.
\end{eqnarray}

\section{Proofs}
\label{proofssection}
So far, we have used the language of special geometry in order to find several derivations in the special polynomial ring generators. This is essentially the main mathematical content of period manipulations in the B-model of mirror symmetry.
In this section we interpret such derivations as vector fields in the moduli space $\T$ introduced in the Introduction 
and hence prove Theorem \ref{maintheo} and Theorem \ref{theo2}.  
We denote an element of $\T$ by $\t$ and hopefully it will not be confused with the special coordinates $t$ of Sec.~\ref{specialcoordinates}. 

\subsection{Generalized period domain}
\label{GPDsection}
The generalized period domain $\pedo$ 
introduced in Ref.~\cite{ho18}, is  the set of all $(2\hn+2)\times (2\hn+2)$-matrices $\per$ with complex
entries which satisfies  $\per^\tr\imh\per=\imc$ and a positivity condition which can be derived from the description bellow. 
Here,  $\imh$ is the symplectic matrix and $\imc$ is given in (\ref{intmatrix}).
The symplectic group $\Sp(2\hn+2,\Z)$ acts on $\pedo$ from the left and the quotient
$$
\ped:=\Sp(2\hn+2,\Z)\backslash \pedo,
$$
parameterizes the set of all lattices $L$ inside $H_\dR^3(X_0)$ such that the data $(L,H_\dR^3(X_0),  F^*_0,\langle \cdot,\cdot \rangle)$
form a polarized Hodge structure. Here,  $X_0$ is a fixed Calabi-Yau threefold as in the Introduction and
$H_\dR^3(X_0), F^*_0,\langle \cdot,\cdot \rangle$ are its de Rham cohomology, Hodge filtration and intersection form,  respectively.
It is defined in such a way that we have the period map
\begin{equation}
\label{13mar2014}
\per :\T\to \ped, \  \  \t\mapsto [\int_{\delta_i}\omega_j]\,,
\end{equation}
where $\t$  now represent the 
pair $(X,\omega)$, $\delta_i$ is a symplectic basis of $H_3(X,\Z)$. 
In order to have a modular form theory attached to the above picture, one has to find a special locus $\uhp$ in $\T$.
In the case of Calabi-Yau threefolds special geometry in 
\S\ref{specialgeometry} gives us the following candidate. 
\begin{defi}\rm
We define $\uhp$ to be the set
of all elements $\t\in\T$ such that $\per(\t)^\tr$ is of the form  (\ref{specialmatrix}). 
\end{defi}
In the case of elliptic curves the set $\uhp$ is biholomorphic to a punctured disc of radius one, see Ref.~\cite{ho14}. In our context we do not
have a good understanding of the global behavior of $\uhp$. This is related to the analytic continuation of periods of Calabi-Yau threefolds.
The set $\uhp$ is neither an algebraic nor an analytic subvariety of $\T$. One can introduce a holomorphic foliation in $\T$ with $\uhp$ as its leaf and
study its dynamics, see for instance \cite{ho06-3} for such a study in the case of elliptic curves. 
For the purpose of $q$-expansions, one only needs to know that a local patch of $\uhp$ is biholomorphic to a complement of a normal crossing
divisor in a small neighborhood of $0$ in $\C^\hn$.

In this abstract context of periods, we think of $t^i$ as $\hn$ independent variables and $F_0$ a function in $t^i$'s.
The restriction of the ring of regular functions in $\T$ to $\uhp$ gives a $\C$-algebra which
can be considered as a generalization of quasi-modular forms. We can do $q$-expansion of such functions around a degeneracy
point of Calabi-Yau threefolds, see Ref.~\cite{ho18} for more details.

\subsection{Proof of Theorem \ref{maintheo}}
\label{06october}
First, we assume that $\k=\C$ and work in
a local patch $U$ of the moduli space of 
of Calabi-Yau threefolds $X_z,\ z\in U$.
Therefore, we have assumed that in $U$ 
the universal family of Calabi-Yau threefolds $X$ exists.
In  \eqref{6oct2014} we have defined $\vec{\omega}_t$ and in Proposition  \ref{18sep2014} we have proved that 
$(X_z, \vec{\omega}_t)\in\uhp$. Therefore, we have the following 
map
\begin{equation}
 \label{davashod}
f: U\to \uhp,\  z\mapsto (X_z, \vec{\omega}_t).
\end{equation}
The entries of the matrix $\Am$ in (\ref{muradalim}) are rational functions
in $z_i,\ i=1,2,\ldots,\hn$ and the generators in Definition \ref{specgen}. 
The Gauss-Manin connection matrix restricted to the image of the map \eqref{davashod} and computed 
in the flat coordinates $t$ is just $
\sum_{i=1}^\hn \tilde\gm_i dt^i$, 
where $\tilde \gm_i$ is the matrix computed  in (\ref{determinantal}).
From this we get the Gauss-Manin connection matrix in the basis $\Am \vec{\omega}_t$:
\begin{equation}
 \label{Iwant}
d\Am\cdot \Am^{-1}+\sum_{i=1}^\hn \tilde \gm_idt^i.
\end{equation}
We consider the generators as variables 
$$
x_1,x_2,\ldots,x_{a},\ \ \ a:=\dim(\BG)
$$
independent of $z_i$'s, and so, we write $\Am=\Am_{z,x}$ and $\tilde \gm_i=\tilde \gm_{i,z,x}$. We get
an open subset $V$ of $U\times \C^a$  such that for $(z,x)\in V$ we have $\Am_{z,x}\in\BG$.  
In this 
way 
$$
\tilde U:=\{ (z, \Am_{z,x}\vec{\omega}_t )\mid (z,x)\in  V \}\,,
$$
is an open subset of $\T$. Using the action of $\BG$ on $\T$,  
the holomorphic limit of polynomial generators can be regarded as holomorphic functions in $\tilde U$. 
Now the Gauss-Manin connection of the enhanced family $\X/\T$ in the open set $\tilde U$ is just
\eqref{Iwant} replacing  $\Am$ and $\tilde \gm_i$ with $\Am_{z,x}$ and $\tilde \gm_{i,z,x}$, respectively.
The existence of the vector fields in Theorem \ref{maintheo} 
follows from the same computations in \S\ref{specialgeometry}. Note that from \eqref{18.09.2014} we get
\begin{eqnarray}
\label{28.09.2014}
\Ra_{\tgtwo_{ab}} &=&  \frac{\partial}{\partial T^{ab}} -\frac{1}{2} (L_a\, \frac{\partial}{\partial T^b} + L_b \frac{\partial}{\partial T^a}) +\frac{1}{2} L_a L_b \frac{\partial}{\partial T}\,,\\ \nonumber
\Ra_{\tgone_{a}}&=&  \frac{\partial}{\partial T^a} - L_a \frac{\partial}{\partial T}\,, \\ \nonumber
\Ra_{\tgzero} &=& \frac{1}{2}\frac{\partial}{\partial T} \,,\\ \nonumber
\Ra_{\kgone^{a}}&=&   \frac{\partial}{\partial L_a}\,,\\ \nonumber
\Ra_{\ggtwo_{b}^{a}}  &=& g^{a}_m \frac{\partial}{\partial g_{m}^{b}} - L_{b} \frac{\partial}{\partial L_a}  + 2 T^{ad} \frac{\partial}{\partial T^{db}} + T^{a} \frac{\partial}{\partial T^b}\,,   \\ \nonumber
\Ra_{\ggzero_0} &=& g_0 \frac{\partial}{\partial g_0} + L_{a} \frac{\partial}{\partial L_a} +T^{a} \frac{\partial}{\partial T^a} + 2 \frac{\partial}{\partial T}\,
\end{eqnarray}
and $\Ra_a$'s are given by the vector fields computed in Proposition \ref{28s2014}.

We now prove the uniqueness. Let us assume that there are two vector fields $\Ra_i,\ i=1,2$  such that $\gm_{\Ra_1}=\gm_{\Ra_2}$
is one of the special matrix format in Theorem \ref{maintheo}. Let us assume that $\Ra=\Ra_1-\Ra_2$ is not zero and so 
it has a non-zero solution $\gamma(y)$ which is a holomorphic map from a neighborhood of $0$ in $\C$ to $\T$ and satisfies 
$\partial_y\gamma=\Ra(\gamma)$. 
We have $\nabla_\Ra\omega_1=0$ and so $\omega_1$  restricted to the image of $\gamma$ is a flat section
of the Gauss-Manin connection. Since $z\to t$ is a coordinate change, we conclude that $z$ as a function of $y$ is constant and
so if $\gamma(y):=(X_y,\{\omega_{1,y},\cdots, \omega_{2\hn+2,y}\})\in\T$ then $X_y$ and $\omega_{1,y}$ do not depend on $y$. 
For the case in which we have the special matrix formats (\ref{gofLie}) we have also $\nabla_{\Ra}\omega_{i,y}=0$ and so $\gamma(y)$
is a constant map. In the case of  \eqref{Someday}, in a similar way all $\omega_{1,y},\omega_{\hn+2,y},\omega_{\hn+3,y},\cdots,\ \omega_{2\hn+2,y}$ 
do not depend on $y$. For others we argue as follows. Since 
$X=X_z$ does not depend on $y$ the differential forms $\omega_{2,y},\omega_{3,y},\cdots,\ \omega_{\hn+1,y}$ are
linear combinations of elements $F^2 H_\dR^3(X)$ (which do not depend on $y$) with coefficients which depend on $y$. 
This implies that the action of $\nabla_\Ra$ on them is still in $F^2 H_\dR^3(X)$. Using \eqref{Someday} we conclude that
they are also independent of $y$.

Now let us consider the algebraic case,  where the universal family $\X/\T$
as in the Introduction exists. We further assume that over a local moduli $U$ 
no Calabi-Yau threefold $X_z,\ z\in U$ has an isomorphism which acts non-identically on $H^3_\dR(X_z)$.
In this way, the total space of choices of $\omega_i$'s over $U$ gives us an
open subset $\tilde\T$ of the moduli space $\T$. The existence of the algebraic 
vector fields $\Ra$ in Theorem \ref{maintheo} is equivalent to the existence of the same 
analytic vector fields in some small open subset of $\T$. This is because to find such vector fields we have to solve 
a set of linear equations with coefficients in $\O_\T$. That is why our argument in the algebraic context cannot guarantee that the vector fields
in Theorem \ref{maintheo} are holomorphic everywhere in $\T$.

\subsection{Proof of Theorem \ref{theo2}}
We use the equalities \eqref{18.09.2014} and \eqref{18.sepetember.2014}
and conclude that the two groups of derivations
$$
\Ra_{\ggtwo^b_{a}}, \Ra_{\tgtwo_{ab}}, \Ra_{\kgone_a}, \Ra_{\tgone_a}, \Ra_{\tgzero}
$$
and 
$$
\frac{\partial}{\partial T^{ab}},
\frac{\partial}{\partial T^{a}},
\frac{\partial}{\partial T},
\frac{\partial}{\partial L_{a}},
\frac{\partial}{\partial g_{m}^{a}} 
$$
are linear combinations of each other. Therefore, if $f\in\O_\T$ is in the left hand side of \eqref{40salnazdik} then its derivation with respect 
to all variables 
$T^{ab},T^{a}, T, L_{a}, g_{m}^{a}$ is zero and so it depends only on $z_i$'s and $g_0$. 
In a similar way we can derive the fact that  
$\bigcap_{\gg\in \Lie(\BG)} \ker(\Ra_{\gg})$ is the set of $\BG$ invariant functions in $\T$.  
\subsection{The Lie Algebra $\liealg$}
In this section we describe the computation of Lie bracket structure of (\ref{29apr2014}) resulting in the Table \ref{liebrackettable}.   
Let $\Ra_1,\Ra_2$ be two vector fields in $\T$ and let $\gm_i:=\gm_{\Ra_i}$.  We have
$$
\nabla_{[\Ra_1,\Ra_2]}\omega=([\gm_2,\gm_1]+\Ra_1(\gm_2)-\Ra_2(\gm_1))\omega $$
In particular, for $\gg_i\in \Lie(\BG)$ we get 
$$
[\Ra_{\gg_1},\Ra_{\gg_2}]=[\gg_1,\gg_2]^\tr.
$$
We have also
$$
[\Ra_i,\Ra_j]=0,\ \ \ i=1,2,\ldots,\hn.
$$
because $\Ra_1(\gm_{\Ra_2})=\Ra_2(\gm_{\Ra_1})$ and $[\gm_{\Ra_1},\gm_{\Ra_2}]=0$. These equalities, in turn, follow from the fact that $\Yuk_{ijk}$ are symmetric in $i,j,k$. 
It remains to compute
$$
[\Ra_{\gg},\Ra_i],\ \ \ i=1,2,\ldots,\hn,\ \ \ \g\in \BG. 
$$
which is done for each element of the canonical basis of $\Lie(\BG)$. 


\subsection{Two fundamental equalities of the special geometry}
\label{twoequalitiessection}
The Gauss-Manin connection matrix $\gm$ satisfies the following equalities:
\begin{eqnarray}
\label{12may2014}
d\gm &= & -\gm\wedge \gm \\
\label{12May2014}
0&=&\gm\imc+\imc\gm^{\tr}.
\end{eqnarray}
The first one follows from the integrability of the Gauss-Manin connection and the second equality follows after taking differential 
of the equality (\ref{intmatrix}). Note that $\imc$ is constant and so $d\imc=0$.  Note also that the base space of our Gauss-Manin connection matrix
is $\T$. For the mirror quintic this is of dimension 7, and so the integrability is a non-trivial statement, whereas the integrability over the classical moduli space of mirror
quintics (which is of dimension one) is a trivial identity. The Lie algebra $\Lie(\BG)$ is already hidden in 
(\ref{12May2014}) and it is consistent with the fact
that after composing $\Ra_{\gg}$ with $\gm$ we get $\gg^\tr$. 
Assuming the existence of  $\Ra_i$'s and $\Yuk_{ijk}$ in Theorem \ref{maintheo}, the equalities (\ref{12may2014}) and (\ref{12May2014})
composed the vector fields $\Ra_i,\Ra_j$ imply that $\Yuk_{ijk}$ are symmetric in $i,j,k$ and the equality (\ref{santana}). 
We want to argue that most of the ingredients of the special geometry can be derived from (\ref{12may2014}) and (\ref{12May2014}). 
Special geometry, in mathematical
terms, aims to find a $\hn$-dimensional sub-locus $(\C^{\hn},0)\cong M\subset \T$ such that $\gm$ restricted to $M$
is of the form
$$
\tilde\gm:= 
\left(
\begin{array}{cccc}
 0 & \omega_1 & 0&0 \\
 0 & 0 & \omega_2 & 0 \\
 0 & 0 & 0 & \omega_3\\
 0 & 0 & 0 & 0 \\
\end{array}
\right)\,.
$$
where the entries of $\omega_i$'s are differential 1-forms in $M$. In \S\ref{GPDsection} the  union of such loci is denoted by $\uhp$.  
The equality (\ref{12May2014}) implies that $\omega_3=\omega_1^\tr$ and $\omega_2=\omega_2^\tr$ and
the equality (\ref{12may2014}) implies that  $\omega_1\wedge \omega_2=0$ and
all the entries of $\omega_i$'s are closed, and since $M\cong(\C^{\hn},0)$, they are exact. Let us write
$\omega_1=dt,\ \omega_2=dP$ for some matrices $t,P$ with entries which are holomorphic functions on $M$ and so we have
\begin{equation}
\label{21june2014}
dt\wedge dP=0.
\end{equation}
Special geometry takes the entries of $t$ as coordinates on $M$ and the equation (\ref{21june2014}) gives us the existence
of a holomorphic function $F$ on $M$ such that  $P=[\frac{\partial F}{\partial t_i\partial t_j}]$. This is exactly the prepotential 
discussed in \S\ref{specialcoordinates}. One can compute the special period matrix $\per$ in  \eqref{specialmatrix}, starting from the initial data
 \begin{equation}
\label{specialmatrix-1}
\per^\tr= \left(\begin{array}{cccc}
1 & * & * & *\\
0&\delta^c_a & *& *\\
0&0&\delta^a_c & * \\
0&0&0&-1
\end{array}
\right)\,,
\end{equation}
and the equality $d\per^\tr=\tilde\gm\per^\tr$, where $\tilde\gm$ is  
the Gauss-Manin connection restricted to $M$.


\section{Mirror quintic case}
\label{mqc}
In Refs.~\cite{ho21, ho22} it was proven that the universal family $\X\to\T$ exists in the case of 
mirror quintic Calabi-Yau threefolds and it is defined over $\Q$. More precisely we have  
$$
\Ts=\spec(\Q[\t_0,\t_4,\frac{1}{(\t_0^5-\t_4)\t_4}]),\ 
$$ 
where for $(\t_0,\t_4)$ we associate the pair $(\X_{\t_0,\t_4},\omega_1)$. In the affine coordinates 
$(x_1,x_2,x_3,x_4)$, that is $x_0=1$,  $\X_{\t_0,\t_4}$ is  given by
\begin{eqnarray*}
 \X_{\t_0,\t_4} &:=& \{ f(x)=0\}/G,\\ 
f(x) &:=& -\t_4-x_1^5-x_2^5-x_3^5-x_4^5+5\t_0x_1x_2x_3x_4,
\end{eqnarray*}
and 
$$
\omega_1:= \frac{ dx_1\wedge dx_2\wedge dx_3\wedge dx_4}{df}.
$$
We have also 
\begin{equation}
 \label{thanksdeligne}
\T=\spec(\Q[\t_0,\t_1,\ldots,\t_6,\frac{1}{(\t_0^5-\t_4)\t_4\t_5}])
\end{equation} 
Here, for $(\t_0,\t_1,\ldots,\t_6)$ we associate the pair $(\X_{\t_0,\t_4},[\omega_1,\omega_2,\omega_3, \omega_4])$, where $\X_{\t_0,\t_4}$ is as before and $\omega$ is given by
\begin{equation}
\label{26apr2011}
\begin{pmatrix}
\omega_1\\
\omega_2\\
\omega_3\\
\omega_4
\end{pmatrix}:=
\begin{pmatrix}
1 & 0& 0 & 0\\
\frac{-5^5\t_0^4-\t_3}{\t_5} & \frac{-5^4(\t_4-\t_0^5)}{\t_5}
 & 0 &0 \\
\frac{(5^5\t_0^4+\t_3)\t_6-(5^5\t_0^3+\t_2)\t_5}{5^4(\t_4-\t_0^5)}
\frac{5^4(\t_0^5-\t_4)}{\t_5}&
\t_6
&
\t_5
&
0 \\
\t_1&\t_2&\t_3& 625(\t_4-\t_0^5)
\end{pmatrix}
\begin{pmatrix}
\omega_1\\
\nabla_{\frac{\partial}{\partial \t_0}}\omega_1 \\
(\nabla_{\frac{\partial}{\partial \t_0}})^{(2)}\omega_1 \\
(\nabla_{\frac{\partial}{\partial \t_0}})^{(3)}\omega_1 \\
\end{pmatrix}
\end{equation}
The above definition is the algebraic counterpart of the equality \eqref{30may04}.
Note that in this context $\t_i$'s are just parameters, whereas the generators of the special polynomial 
differential ring are functions in a local patch of the classical moduli space of Calabi-Yau threefolds.
The relation between these two sets are explained in \S\ref{06october}. 
The genus one topological string partition function  $F_1^A$ is given by
\begin{equation}
\label{07feb2014}
\Fg_1^\alg:=-\ln(\t_4^{\frac{5}{12}}(\t_4-\t_0^5)^{ \frac{-1}{12}}\t_5^{\frac{1}{10}})
\end{equation}
and for $g\geq 2$,  we have
\begin{equation}
 \Fg_g^\alg =\frac{Q_g}{(\t_4-\t_0)^{2g-2}\t_5^{3g-3}},\
\end{equation}
where $Q_g$ is homogeneous polynomial of degree $69(g-1)$ with weights 
\begin{equation}
\label{30sept}
\deg(\t_i):=3(i+1),\ i=0,1,2,3,4, \ \ \deg(\t_5):=11,\ \  \deg(\t_6):=8.
 \end{equation}
and with rational coefficients, 
and so $\Fg_g^\alg$ is of degree $6g-6$. 
Further, any monomial $\t_0^{i_0}\t_1^{i_1}\cdots \t_6^{i_6}$ in $\Fg_g^\alg$ satisfies $i_2+i_3+i_4+i_5+i_6\geq 3g-3$. 
For the $q$-expansion of  $\t_i$'s see Ref.~\cite{ho22}. We have
\begin{equation}
\label{yukawacoupling}
\Yuk_{111}=\frac{5^8(\t_4-\t_0)^2}{\t_5^3}
\end{equation}
and 
{\tiny 
\begin{eqnarray}
\Ra_1 &=& \frac{3750\t_{0}^{5}+\t_{0}\t_{3}-625\t_{4}}{\t_{5}} \frac{\partial}{\partial \t_0} 
-\frac{390625\t_{0}^{6}-3125\t_{0}^{4}\t_{1}-390625\t_{0}\t_{4}-\t_{1}\t_{3}}{\t_{5}} \frac{\partial}{\partial \t_1 } 
\\ 
& & 
-\frac{5859375\t_{0}^{7}+625\t_{0}^{5}\t_{1}-6250\t_{0}^{4}\t_{2}-5859375\t_{0}^{2}\t_{4}-625\t_{1}\t_{4}-2\t_{2}\t_{3}}{\t_{5}} \frac{\partial}{\partial \t_2 }\nonumber \\ \nonumber
& & -\frac{9765625\t_{0}^{8}+625\t_{0}^{5}\t_{2}-9375\t_{0}^{4}\t_{3}-9765625\t_{0}^{3}\t_{4}-625\t_{2}\t_{4}-3\t_{3}^{2}}{\t_{5}} \frac{\partial}{\partial \t_3}
\\ \nonumber & &
+\frac{15625\t_{0}^{4}\t_{4}+5\t_{3}\t_{4}}{\t_{5}} \frac{\partial}{\partial \t_ 4} 
-\frac{625\t_{0}^{5}\t_{6}-9375\t_{0}^{4}\t_{5}-2\t_{3}\t_{5}-625\t_{4}\t_{6}}{\t_{5}} \frac{\partial}{\partial \t_ 5} \\ \nonumber  & & 
+\frac{9375\t_{0}^{4}\t_{6}-3125\t_{0}^{3}\t_{5}-2\t_{2}\t_{5}+3\t_{3}\t_{6}}{\t_{5}}\frac{\partial}{\partial \t_ 6}
\\ \nonumber
\Ra_{\ggtwo^{1}_{1}} &=& 
 \t_{5} \frac{\partial}{\partial \t_5 }+ \t_{6}\frac{\partial}{\partial \t_6 },
 \\ \nonumber
 \Ra_{ \ggzero_0}  &= & \t_{0} \frac{\partial}{\partial \t_0 }+ 2\t_{1} \frac{\partial}{\partial \t_1 } 
+3\t_{2} \frac{\partial}{\partial \t_2 } +4\t_{3} \frac{\partial}{\partial \t_3 }+ 5\t_{4} \frac{\partial}{\partial \t_4 }+
3\t_{5} \frac{\partial}{\partial \t_5 }+ 2\t_{6}\frac{\partial}{\partial \t_6 }\,,
\\ \nonumber
\Ra_{ \kgone_{1} } &=&
 -\frac{5\t_{0}^{4}\t_{6}-5\t_{0}^{3}\t_{5}-\frac{1}{625}\t_{2}\t_{5}+\frac{1}{625}\t_{3}\t_{6}}{\t_{0}^{5}-\t_{4}} 
 \frac{\partial}{\partial \t_ 1}+ \t_{6} \frac{\partial}{\partial \t_2 }+ \t_{5} \frac{\partial}{\partial \t_3 },\ \ \ \ \ 
 \\ \nonumber
 \Ra_{ \tgtwo_{11}} &=&
 \frac{625\t_{0}^{5}-625\t_{4}}{\t_{5}}\frac{\partial}{\partial \t_6 } \,,
 \\ \nonumber
 \Ra_{ \tgone_{1}} &=&  \frac{-3125\t_0^4-\t_3}{\t_5}\frac{\partial }{\partial \t_1}+  
 \frac{625(\t_0^5-\t_4)}{\t_5}\frac{\partial }{\partial \t_2}\,,
\\ \nonumber
\Ra_{\tgzero} &=& \frac{\partial }{\partial \t_1}\,.
\end{eqnarray}
}
Using the asymptotic behavior of $\Fg_g^\alg$'s in Ref.~\cite{Bershadsky:1993ta}, we know that the ambiguities of $\Fg_g$ arise from 
the coefficients of 
\begin{equation}
 \label{4mar2014}
\frac{P_g(\t_0,\t_4)}{(\t_4-\t_0^5)^{2g-2}},\ \ \ \deg(P_g)=36(g-1)\,.
\end{equation}
Knowing that we are using the weights (\ref{30sept}), we observe that it
depends on $[\frac{12(g-1)}{5}]+1$ coefficients. 
The monomials in \eqref{4mar2014} are divided into two groups,  those meromorphic in $\t_4-\t_0^5$ and the rest which is
$$
\t_0^a(\t_4-\t_0^5)^b,\ \ \ a+5b=2g-2,\ a,b\in\N_0.
$$
The coefficients of the first group can be fixed by the so called gap condition and the asymptotic behavior of $F_g^A$ at the conifold,
see for instance Ref.~\cite{Huang:2006hq}. One of the coefficients in the second group  can be solved using the 
asymptotic behavior of $\Fg_g^\alg$ at the maximal unipotent monodromy point. In total, we have $[\frac{2g-2}{5}]$ undetermined coefficients. It is not clear whether it is possible to solve them using only the data attached to the mirror quintic Calabi-Yau threefold. Using the generating function role that $\Fg_g^\alg$'s has on the $A-$model side for
counting curves in a generic quintic, one may solve all the ambiguities given enough knowledge of enumerative invariants, such computations 
are usually hard to perform, see for instance Ref.~\cite{Huang:2006hq} for a use of boundary data in the B-model and $A-$model counting data which determines the ambiguities up to genus $51$.

\section{Final remarks}
\label{remarkssection}
We strongly believe that a mathematical verification of mirror symmetry
at higher genus will involve the construction of the Lie algebra $\liealg$ in the $A$-model
Calabi-Yau variety $\check X$.
The genus zero case was established by Givental \cite{Givental:1996} and Lian, Liu and Yau \cite{Lian:1997} for many cases of Calabi-Yau threefolds and in particular the quintic case, that is, the period
manipulations of the $B$-model lead to the virtual number of rational 
curves in the $A$-model. The genus one case was proved by Zinger in Ref.~\cite{Zinger:2009}. 
The amount of computations and technical difficulties from genus zero to genus one case is significantly large. For higher genus there has been no 
progress. The period expressions involved in higher genus, see for instance Ref.~\cite{Yamaguchi:2004bt},  are usually huge and this is the main reason
why the methods used in Refs.~\cite{Givental:1996,Lian:1997, Zinger:2009} do not generalize. This urges us for an 
alternative description of the generating function of the number of higher genus curves in the $A$-model. 
The original  formulation in Ref.~\cite{Bershadsky:1993cx} using holomorphic anomaly equation for genus $g$ topological string partition functions,
is completely absent in the mathematical formulation of $A$-model using quantum differential equations. Motivated by this, we formulated
Theorem \ref{maintheo} and the algebraic holomorphic anomaly equation \eqref{aaeq}. 
Our work opens many other new conjectures in the $A$-model Calabi-Yau varieties. The most significant one is the following.
Let $\YukA_{ijk}^A$ and $\Fg_g^A,\ g\geq 1$ 
be the generating function of genus zero and genus $g$ Gromov-Witten
invariants of the $A$-model Calabi-Yau threefold $\check X$, respectively
\begin{conj}\rm
Let $\check X$ be a Calabi-Yau threefold with $\hn:=\dim(H^2_\dR(X))$ and let  $M$ be the sub-field 
of formal power series  generated by $\YukA_{ijk}^A,\exp(\Fg^{A}_1), \Fg_g^A,\ \ g\geq 2$
and their derivations under $q_i\frac{\partial }{\partial q_i},\ \ i=1,2,\ldots,\hn$. The transcendental degree of $M$ over
$\C$ is at most $a_{\hn}:=\frac{3\hn^2+7\hn+4}{2}$, that is, for any $a_{\hn}+1$ elements $x_1,x_2,\ldots, x_{a_\hn+1}$ of $M$ there is
a polynomial $P$ in $a_{\hn}+1$ variables and with coefficients in $\C$ such that $P(x_1,x_2,\ldots, x_{a_{\hn}+1})=0$.
\end{conj}
The number $a_\hn$ is the dimension of the moduli space $\T$ in the Introduction. Our mathematical knowledge in enumerative
algebraic geometry of Calabi-Yau threefolds is still far from any solution to the above conjecture. 

Our reformulation of BCOV anomaly equation opens an arithmetic approach to Topological String partition functions. 
For many interesting example such as mirror quintic, $\T$ can be realized as an affine scheme over $\Z[\frac{1}{N}]$ 
for some integer $N$.
In this way we can do mod primes of $\Fg_g^\alg$'s which might give some insight into the arithmetic of 
Fourier expansions of $\Fg_g^\alg$'s.

In the case of mirror quintic we have partial compactifications of $\T$ given by $\t_4=0$, $\t_4-\t_0^5=0$ and $\t_5=0$.
The first two correspond to the maximal unipotent and conifold singularities. The degeneracy locus
$\t_5=0$ corresponds to degeneration of differential forms and not the mirror quintic itself. 
Our computations  show that the vector fields in Theorem \ref{maintheo} are holomorphic everywhere except $\t_5$. These statements cannot be seen for the proof of 
Theorem \ref{maintheo} and one may conjecture that similar statement in general must be valid. Of course one must first construct
the universal family $\X/\T$ and enlarge it to a bigger family using similar moduli spaces for limit mixed Hodge structures.



\def\cprime{$'$} \def\cprime{$'$} \def\cprime{$'$}

\newcommand{\etalchar}[1]{$^{#1}$}

Murad Alim \\  \emph{\small Mathematics Department, Harvard University, 1 Oxford Street, Cambridge MA, 02138 USA. \\and Jefferson Physical Laboratory, 17 Oxford Street, Cambridge, MA, 02138, USA}\\
{\small alim@math.harvard.edu}
\\
\\
Hossein Movasati\\
\emph{\small Mathematics Department, Harvard University, 1 Oxford Street, Cambridge MA, 02138 USA. \\
On sabbatical leave from Instituto de Matem\'atica Pura e Aplicada, IMPA, Estrada Dona Castorina, 110, 22460-320, Rio de Janeiro, RJ, Brazil.}\\
{\small hossein@impa.br}
\\
\\
Emanuel Scheidegger \\
\emph{\small Mathematisches Institut, Albert-Ludwigs-Universit\"at Freiburg,\\ \small Eckerstrasse 1, D-79104 Freiburg, Germany.}\\
{\small emanuel.scheidegger@math.uni-freiburg.de}
\\
\\
Shing-Tung Yau \\
\emph{\small Mathematics Department, Harvard University, 1 Oxford Street, Cambridge MA, 02138 USA.}
{\small yau@math.harvard.edu}

\end{document}